\documentclass[11pt, reqno]{amsart}
\usepackage{enumerate,amsmath,amssymb, mathrsfs}
\usepackage{xcolor}

\allowdisplaybreaks[4]

\usepackage{latexsym, amssymb,enumerate}
\usepackage{appendix}

\newcommand{\be}{\begin{equation}}
\newcommand{\ee}{\end{equation}}
\newcommand{\beq}{\begin{eqnarray}}
\newcommand{\eeq}{\end{eqnarray}}

\def\H{{\mathbb H}}

\def\R{{\mathfrak R}}

\newtheorem{prop}{Proposition}[section]
\newtheorem{theo}[prop]{Theorem}
\newtheorem{lemm}[prop]{Lemma}
\newtheorem{coro}[prop]{Corollary}
\newtheorem{rema}[prop]{Remark}

\def\begeq{\begin{equation}}
\def\endeq{\end{equation}}
\def\p{\partial}

\def\R{\mathbb R}

\def\<{\langle}
\def\>{\rangle}
\def\d{\delta}

\def\s{\sigma}

\def\v{\varepsilon}
\def \ds{\displaystyle}

\def\S{\mathbb  {S}}

\makeatletter
\@namedef{subjclassname@2020}{\textup{2020} Mathematics Subject Classification}
\makeatother

\def\odot{\setbox0=\hbox{$\bigcirc$}\relax \mathbin {\hbox
to0pt{\raise.5pt\hbox to\wd0{\hfil $\wedge$\hfil}\hss}\box0 }}

\numberwithin{equation} {section}

\def\tilde{\widetilde}

\begin{document}

\title[New weighted Alexandrov-Fenchel type inequalities]{New weighted Alexandrov-Fenchel  type inequalities  and Minkowski inequalities in space forms}



\author{Jie Wu}
\address{School of Mathematical Sciences, Zhejiang University, Hangzhou 310058, P. R. China}
\email{wujiewj@zju.edu.cn}

\begin{abstract} In this paper, we establish a broad class of new sharp Alexandrov-Fenchel inequalities involving general convex weight functions for static convex hypersurfaces in hyperbolic space.
Additionally, we  derive new weighted Minkowski-type inequalities for static convex hypersurfaces in hyperbolic space $\H^n$ and for convex hypersurfaces in the sphere $\S^n$. The tools we shall use are the locally constrained  inverse curvature flows in hyperbolic space and in the sphere.
\end{abstract}

\maketitle

\

\section{Introduction}

The Alexandrov-Fenchel inequalities for  quermassintegrals of convex domains  play an important role in classical geometry.
In Euclidean space, the celebrated Alexandrov-Fenchel inequalities \cite{AF1,AF2} for convex  hypersurfaces $\Sigma \subset\R^n$ state that
\begin{equation}\label{eq01}
\frac 1{\omega_{n-1}}\int_{\Sigma} H_kd\mu \ge \left(\frac 1{\omega_{n-1}}\int_{\Sigma} H_{l} d\mu \right)^{\frac {n-1-k}{n-1-l}}, \quad 0\le l< k\le  n-1,
\end{equation}
 where $H_k$ denotes the normalized $k$-th mean curvature of $\Sigma$ and $\omega_{n-1}$ is the area of the unit sphere $\S^{n-1}$. Equality in (\ref{eq01}) holds if and only if $\Sigma$ is a sphere.
When $k=1,\, l=0$, (\ref{eq01}) reduces to
\begin{equation}\label{Mineq}
\frac 1{\omega_{n-1}}\int_{\Sigma} H_1d\mu \ge \left(\frac {|\Sigma|}{\omega_{n-1}} \right)^{\frac {n-2}{n-1}},
\end{equation}
which is usually called the Minkowski inequality. The Alexandrov-Fenchel inequalities (\ref{eq01}) have been extended to certain classes of non-convex hypersurfaces (see \cite{CW, GuanLi,Q} for instance).

To establish the Alexandrov-Fenchel type inequalities for quermassintegrals in space forms $N^n(\v)$, we consider the simply connected space forms of constant sectional curvature $\v=-1,0$ or $1$, which are hyperbolic space $\H^n$, Euclidean space $\R^n$, and the sphere $\S^n$, respectively. These space forms can be viewed as a warped product manifold $I\times \S^{n-1}$ equipped with the metric
$$\bar g=dr^2+\lambda(r)^2 g_{\S^{n-1}},$$
  where $g_{\S^{n-1}}$ denotes the standard round metric on $\S^{n-1}\subset \R^n$. Specifically, $N^n(-1)$ corresponds to  hyperbolic space $\H^n$ with
 $\lambda(r)=\sinh r,$ where $r\in[0, \infty) $; $N^n(0)$ corresponds to Euclidean space $\R^n$ with  $\lambda(r)=r,$ where  $r\in[0, \infty) $; and $N^n(1)$ corresponds to the sphere $\S^n$ with $\lambda(r)=\sin r,$ where $r\in[0, \pi).$

For any smooth convex domain $\Omega$ in the space form $N^n(\v)$ with boundary $\p\Omega=\Sigma$,
the $k$-th  quermassintegral  $W_k(\Omega)$ can be interpreted as the measure of the set of $k$-dimensional totally geodesic subspaces which intersect $\Omega$ in integral geometry \cite{San}.  In $\R^n$,  quermassintegrals coincide with curvature integrals up to a constant multiple.  In $\H^n$ and $\S^n$, however, quermassintegrals and curvature integrals differ but remain closely related through the following  recursive formulas (see e.g. \cite[Proposition 7]{Solanes}):
 \begin{eqnarray}
 &&W_0(\Omega)=\hbox{Vol}(\Omega), \quad W_1(\Omega)=\frac 1 {n-1}|\Sigma|,\quad W_n(\Omega)=\frac{1}{n}\omega_{n-1}, \nonumber\\
 &&W_{k+1}(\Omega)=\frac{1}{n-1-k}\int_{\Sigma}H_k d\mu+\v\frac{k}{n-1-k}W_{k-1}(\Omega),\quad k=1, \cdots, n-2,\label{quermass}
 \end{eqnarray}
where $H_k$ denotes the normalized $k$-th mean curvature of $\Sigma$.
A key feature of  quermassintegrals in $N^n(\v)$ is their variational property, analogous to curvature integrals in $\R^n$:
\begin{equation}\label{vari}
\frac{d}{dt}W_k(\Omega_t)=\int_{\partial\Omega_t} FH_k d\mu_t, \quad k=0, \cdots, n-1,
\end{equation}
for any normal variation with speed $F$.

 In hyperbolic space $\H^n$,  Wang and Xia \cite{WX} established the following Alexandrov-Fenchel type inequality for h-convex domains (i.e., domains whose boundary $\partial\Omega$ satisfies $\kappa_i\geq 1$ for all principal curvatures $\kappa_i$), using a quermassintegral preserving curvature flow:
\begin{equation}\label{HAF}
W_k(\Omega)\geq f_k\circ f_l^{-1}(W_l(\Omega)), \quad 0\leq l<k\leq n.
\end{equation}
with equality if and only if $\Omega$ is a geodesic ball. Here, $f_k(r)=W_k(B_r)$, denotes the $k$-th quermassintegral of the geodesic ball $B_r$ of radius $r$,  and $f_l^{-1}$ is the inverse function of $f_l$.  Inequality (\ref{HAF}) with $k=3$ and $l=1$ was proved earlier by Li, Wei and Xiong \cite{LWX} for star-shaped and $2$-convex hypersurfaces (i.e., $H_1, H_2>0$).  Ge, Wang  and the author of this paper \cite{GWW1} extended (\ref{HAF})
to $k=2m+1 (0<2m<n-1)$ and $l=1$ for $h$-convex domains. Their work derived an optimal Sobolev type inequality for the Gauss-Bonnet curvature $L_k$, which implies (\ref{HAF}) for the odd $k$. Hu, Li and Wei \cite{HLW} provided an alternative proof of (\ref{HAF})  using Brendle-Guan-Li's flow \cite{BGL}.  There have been significant efforts to relax the  h-convexity requirement in (\ref{HAF}) by considering weaker geometric conditions. For recent progress in this direction, we refer the readers to \cite{ACW, AHL, BGL, HL} and related works.

Recently, there have been extensive interest in establishing  weighted geometric inequalities in space forms $N^n(\v)$,
 comparing weighted curvature integrals to quermassintegrals. These inequalities
 generalize the classical Alexandrov-Fenchel inequalities and are categorized into two classes: those weighted by $\lambda'(r)$, derived from the radial derivative of the warped product function $\lambda(r)$, and those weighted by $\Phi(r)$. The primitive function $\Phi(r)$ is explicitly given by 
\begin{equation}\label{Phi}
\Phi(r)=\int_0^r \lambda(s)ds=
\left\{\begin{aligned}
 & \cosh r-1, \quad \v=-1, \\
 & \;\,\frac {r^2} 2, \qquad\qquad \v=0,\\
&  1-\cos r, \qquad \v=1,
\end{aligned} \right.
\end{equation}
corresponding to hyperbolic space $\H^n$, Euclidean space $\R^n$, and the sphere $\S^n$, respectively.
  Notably, in Euclidean space $\R^n$, the weight $\lambda'(r)=1$, as $\lambda(r)=r$. Consequently, the weighted inequalities with $\lambda'(r)$ reduce to the classical unweighted Alexandrov-Fenchel inequalities.

In hyperbolic space $\H^n$, the Alexandrov-Fenchel inequalities  weighted with the factor $\lambda'(r)=\cosh r$ are closely related to the Penrose inequality for asymptotically hyperbolic graphs, see for example,  \cite{DGS, deLG, GWW2, GWWX} and related references. For $k=1$,  Brendle-Hung-Wang \cite{BHW}
established the following Minkowski type inequality
\begin{equation}\label{eq03}
    \int_{\Sigma}\bigg(\lambda'(r) H_1-\langle\bar{\nabla}\lambda'(r),\nu\rangle\bigg)d\mu\geq \omega_{n-1}^{\frac 1{n-1}}{|\Sigma|}^{\frac{n-2}{n-1}},
\end{equation}
and  de Lima and Gir\~ao \cite{deLG}  proved the following related inequality
\begin{equation}\label{eq04}
    \int_{\Sigma}\lambda'(r) H_1 d\mu\geq \omega_{n-1}\left(\left(\frac {|\Sigma|}{\omega_{n-1}}\right)^{\frac{n-2}{n-1}}+\left(\frac {|\Sigma|}{\omega_{n-1}}\right)^{\frac{n}{n-1}}\right),
\end{equation}
 if $\Sigma$ is star-shaped and mean convex (i.e. $H_1>0$).  Later, Ge, Wang and the author \cite{GWW2} generalized this to a family of  higher-order weighted geometric inequalities for  h-convex hypersurfaces,
 \begin{equation}\label{eq1t}
  \int_{\Sigma}\lambda'(r) H_{2k+1} d\mu \ge \omega_{n-1}{\left(\left(\frac{|\Sigma|}{\omega_{n-1}}\right)^{\frac{n}{(k+1)(n-1)}}+\left(\frac{|\Sigma|}{\omega_{n-1}}\right)^{\frac{n-2k-2}{(k+1)(n-1)}} \right)}^{k+1}.
\end{equation}
As an application, we proved a
positive mass theorem for the Gauss-Bonnet-Chern mass of asymptotically hyperbolic graphs.
 (\ref{eq1t}) was  extended by Hu, Li and Wei \cite{HLW} for h-convex hypersurfaces in  hyperbolic space $\H^n$ as follows:
\begin{equation}\label{hlw}
\int_{\Sigma}\lambda'(r) H_k d\mu\geq \phi_{k,l}(W_l(\Omega)),\quad 0\leq l\leq k\leq n-1,
\end{equation}
with equality  if and only if $\Sigma$ is a geodesic sphere centered at the origin. Here $\phi_{k,l}$ denotes the unique strictly increasing function such that equality holds in (\ref{hlw}) for geodesic spheres.

Now, we turn to the study of Alexandrov-Fenchel type inequalities weighted with the factor $\Phi(r)$ in space forms $N^n(\v)$.
In Euclidean space $\R^n$, Wei and Zhou \cite{WZ} proved for star-shaped and $k$-convex  hypersurfaces (i.e., $H_j>0$,\, $\forall \,1\leq j\leq k$)
\begin{equation}\label{Wei1}
       \int_{\Sigma}\Phi H_k d\mu \geq \frac 12\omega_{n-1}\bigg(\frac{\int _{\Sigma} H_{l}d\mu}{\omega_{n-1}}\bigg)^{\frac{n+1-k}{n-1-l}}, \quad 0\leq l<k\leq n-1,
     \end{equation}
where $\Phi=\frac 12 r^2$ as defined in (\ref{Phi}).
Equality holds in (\ref{Wei1}) if and only if $\Sigma$ is a geodesic sphere centered at the origin. This result
generalizes earlier work in \cite{GR,KM1, KM2}.
Recently, the author of this paper \cite{Wu} extended this result to a broader weighted class involving  $\Phi^{\alpha}$ for any real number $\alpha\geq 1$, proving that for any smooth closed star-shaped and $k$-convex hypersurface $\Sigma$ in $\R^n,$
\begin{equation}\label{W1}
       \int_{\Sigma}\Phi^{\alpha} H_k d\mu\geq \frac 1 {2^{\alpha}}\omega_{n-1}\bigg(\frac{\int _{\Sigma} H_{l}d\mu}{\omega_{n-1}}\bigg)^{\frac{n-1-k+2\alpha}{n-1-l}}, \quad 0\leq l<k\leq n-1.
     \end{equation}
The proof of (\ref{W1}) utilizes the normalized(rescaled) inverse curvature flow
\begin{equation}\label{flow}
\frac{\partial}{\partial t} X=\bigg(\frac {H_{k-1}}{H_k}-u\bigg)\nu,
\end{equation}
introduced independently by Gerhardt \cite{Gerhardt} and Urbas \cite{Urbas}. The key step involves establishing  the monotonicity of $\int_{\Sigma} \Phi^{\alpha} H_k$ along the flow (\ref{flow1}) for all $\alpha\geq 1,$ which relies crucially on an iteration formula for the Newton operator
$$T_k=\sigma_k I-T_{k-1}h, $$
where $h$ denotes the second fundamental form.
Notably, this iterative relationship for $T_k$ appears to be employed for the first time in establishing the monotonicity along curvature flows.
By the same method, Kwong and Wei \cite{KW2} obtained a further generalization of (\ref{W1}), replacing  the weight $\Phi^{\alpha}$ ($\alpha\geq 1$) with $g(\Phi)$ for any non-decreasing, convex, $C^2$ positive function $g$. Specifically, they proved
\begin{equation}\label{KW}
\int_{\Sigma} g(\Phi) H_k d\mu\geq \chi_k\circ \xi_l^{-1}(W_l(\Omega)), \quad 0\leq l\leq k\leq n-1,
\end{equation}
where $\Phi=\frac 12 r^2$,  $\chi_k(r)=\int_{S^{n-1}(r)}g(\Phi)H_k d\mu$ and $\xi_l(r)=\omega_{n-1}r^{n-l}$ are both strictly increasing functions on $\R_{+},$ with $\xi_l^{-1}$  denoting the inverse of $\xi_l(r).$

In hyperbolic space $\H^n$, Wei and Zhou \cite{WZ} proved for h-convex hypersurfaces,
\begin{equation}\label{Wei2}
       \int_{\Sigma}\Phi H_k d\mu \geq \xi_{k,l} (W_l(\Omega)), \quad 0\leq l\leq k\leq n-1.
     \end{equation}
 with equality  if and only if $\Sigma$ is a geodesic sphere centered at the origin. Here $\xi_{k,l}$ denotes the unique strictly increasing function such that equality holds in (\ref{Wei2}) for geodesic spheres.  For a related work, we refer to \cite{KW1}.

Inspired by the above works, it is natural to ask whether there is a family of Alexandrov-Fenchel type inequalities  with more general weight funciton
$f(\Phi)$ in  hyperbolic space $\H^n$.
In this paper, we first establish the following Alexandrov-Fenchel quermassintegral inequalities in $\H^n$ with a general convex weight for static convex hypersurfaces.
Here a bounded domain $\Omega$ in $\H^n$ with smooth boundary $\p\Omega=\Sigma$,  is called \textit{static convex}  if its second fundamental form satisfies
$$h_{ij} > \frac{u}{\lambda'(r)} g_{ij}>0, \quad \mbox{everywhere on} \;\Sigma,$$
where $\lambda'(r)=\cosh r$, $\nu$  is the unit outward normal,
and $u=\langle\lambda (r)\partial r, \nu\rangle$ is the support function of the hypersurface $\Sigma$.
This definition was first introduced by Brendle and Wang \cite{BW}. We remark that static convexity implies the strict convexity, but is weaker than $h$-convexity as $\frac{u}{\lambda'(r)}<1$ necessarily holds in $\H^n$.

\begin{theo}\label{thm1.1}
Let  $\Omega$ be a bounded  static convex domain with smooth boundary $\Sigma$ in $\H^n$ .
Then for  all $k = 2,\cdots, n-1,$ and any non-decreasing, convex $C^2$ positive function $f$
which satisfies
\begin{equation}\label{f condi}
f'(s)\geq \frac{k}{k-1}\frac{f(s)}{s},
\end{equation}
we have
\begin{equation}
\int_{\Sigma} f(\Phi) H_k d\mu \geq \chi_k\circ f_l^{-1} \left(W_l(\Omega)\right), \quad  0\leq l\leq k, \label{W2}
\end{equation}
and
\begin{equation}\label{W0}
\int_{\Sigma} f(\Phi) H_k d\mu \geq \chi_k\circ h^{-1} \left(\int_{\Omega}\lambda'(r)\, dv\right),
\end{equation}
where $\Phi=\cosh r -1$ as defined in (\ref{Phi}), $\chi_k(r):=\int_{\partial B_r} f(\Phi) H_k d\mu$ on the geodesic ball $B_r$ of radius $r$, $f_l(r)=W_l(B_r),$ denotes the $l$-th quermassintegral of the geodesic ball $B_r$, $f_l^{-1}$ is the inverse function of $f_l$, and $h(r)=\int_{B_r}\lambda'(r) dv=\frac 1 n\omega_{n-1}(\sinh r)^{n}$.
Moreover, equality in (\ref{W2}) (\eqref{W0} resp.) holds if and
only if $\Omega$ is a geodesic ball centered at the origin.
\end{theo}
Theorem 1.1 generates a broad class of geometric inequalities, where each convex non-decreasing positive function satisfying \eqref{f condi} induces an associated  inequality. First,  observe that condition \eqref{f condi} is equivalent to the nonnegativity of the derivative
\begin{equation}\label{f'}
\left(\frac{f(s)}{s^{\frac{k}{k-1}}}\right)'\geq 0.
\end{equation}
Consequently, a natural class of functions $f$ satisfying all the hypotheses in Theorem 1.1 is given by
\begin{equation}
f(s)=g(s)s^{\alpha},\qquad \alpha\geq \frac{k}{k-1},
\end{equation}
where $g$ is any non-decreasing, convex, $C^2$ positive function. In particular, choosing $g(s)\equiv 1$, yields the following corollary, which may be interpreted as the hyperbolic counterpart of the weighted Alexandrov-Fenchel inequality stated in \eqref{W1}.

\begin{coro}
Let  $\Omega$ be a  bounded static convex domain  with smooth boundary $\Sigma$ in $\H^n$.
Then for all $k = 2,\cdots, n-1,$ and any real number $\alpha\geq \frac{k}{k-1}$, we have
\begin{equation}\label{W2'}
\int_{\Sigma} {\Phi}^{\alpha} H_k d\mu \geq \psi_k\circ f_l^{-1} \left(W_l(\Omega)\right), \quad 0\leq l\leq k,
\end{equation}
and
\begin{equation}
\int_{\Sigma} {\Phi}^{\alpha} H_k d\mu \geq \psi_k\circ h^{-1} \left(\int_{\Omega}\lambda'(r)\, dv\right), \label{W0'}
\end{equation}
where $\Phi=\cosh r -1,$  $\psi_k(r):=\int_{\partial B_r} {\Phi}^{\alpha} H_k d\mu$ on the geodesic ball $B_r$ of radius $r$, $f_l(r)=W_l(B_r),$ $f_l^{-1}$ is the inverse function of $f_l$ and $h(r)=\int_{B_r}\lambda'(r) dv=\frac 1 n\omega_{n-1}(\sinh r)^{n}$.
 Moreover, equality in \eqref{W2'} (\eqref{W0'} resp.) holds if and only if $\Sigma$ is a geodesic sphere centered at the the origin.
\end{coro}

A direct consequence of  (\ref{W2}) and (\ref{W0}) is the following isoperimetric type result:
\begin{coro}
Among the class of static convex hypersurfaces with fixed quermassintegrals $W_l(\Omega)$ or fixed weighted volume $\int_{\Omega}\lambda'(r) dv$, where $f$ is a convex and non-decreasing positive function satisfying \eqref{f condi},
the minimum of weighted curvature integral $\int_{\Sigma}f(\Phi) H_k d\mu$ $(k\geq l)$ is achieved by and only by the geodesic spheres centered at the origin.
\end{coro}
The proof of Theorem $1.1$ employs  the following locally constrained inverse curvature flow (first studied by Scheuer and Xia \cite{SX}):
\begin{equation}\label{flow1}
\frac{\partial}{\partial t} X=\bigg(\frac {H_{k-1}}{H_k}-\frac{u}{\lambda'(r)}\bigg)\nu,
\end{equation}
where $\lambda'(r)=\cosh r$, $\nu$  is the unit outward normal,
and $u=\langle\lambda (r)\partial r, \nu\rangle$ is the support function of the hypersurface $\Sigma$.
It was shown in \cite[Theorem 1.6]{HL} that if the initial hypersurface is static convex,  the solution  $\Sigma_t$ remains static convex for all $t>0.$ Moreover, $\Sigma_t$ converges to a geodesic  sphere centered at the origin as $t\rightarrow \infty.$
A pivotal step in our argument involves proving that the weighted curvature integral
\begin{equation}\label{func}
\int_{\Sigma} f(\Phi) H_k d\mu,
\end{equation}
is monotone non-increasing along the  flow (\ref{flow1}) (see Theorem 4.1 below).
On the other hand, the quermassintegrals $W_l(\Omega)$ for $(l\leq k)$ and weighted volume $\int_{\Omega}\lambda'(r) dv$ are non-decreasing under the flow (\ref{flow1}) due to the variation formula \eqref{vari} and the Heintze-Karcher type inequality.
These monotonicity properties together imply the geometric inequalities in Theorem 1.1.

Since $\Phi=\lambda'(r)-1$ in hyperbolic space, the inequality \eqref{Wei2} improves  the inequality \eqref{hlw}, as noted in the Remark 1.4 of \cite{WZ}. Additionally, we emphasize that Theorem 1.1 does not include the case $k=1$.
In the second part of this paper, we  establish the following weighted Minkowski-type inequalites for static convex hypersurfaces in hyperbolic space $\H^n$ and for convex hypersurfaces in the sphere $\S^n$.
\begin{theo}\label{MH}
Suppose that $\Sigma$ is a smooth static convex hypersurface in $\H^n$ enclosing a bounded domain $\Omega$ that contains the origin. Then for any non-decreasing convex $C^2$ positive function $f$ which satisfies
\begin{equation}
\left(\frac{f(s)}{s}\right)'\geq 0,
\end{equation}
we have
\begin{equation}\label{W3}
\int_{\Sigma} f(\Phi) H_1 d\mu-\int_{\Omega}f(\Phi)dv\geq \chi\circ \xi^{-1} (|\Sigma|),
\end{equation}
and
\begin{equation}\label{W3'}
\int_{\Sigma} f(\Phi) H_1 d\mu-\int_{\Omega}f(\Phi)dv\geq \chi\circ h^{-1} \left(\int_{\Omega}\lambda'(r) \, dv\right),
\end{equation}
where $\chi(r)=\int_{\p B_r}f(\Phi)H_1 d\mu-\int_{B_r}f(\Phi)dv$,  $\xi(r)=\omega_{n-1}\lambda^{n-1}(r)=\omega_{n-1}(\sinh r)^{n-1}$, and $h(r)=\int_{B_r}\lambda'(r) dv=\frac 1 n\omega_{n-1}(\sinh r)^{n}$.
Moreover, equality in (\ref{W3}) (\eqref{W3'} resp.) holds if and
only if $\Omega$ is a geodesic ball centered at the origin.
\end{theo}

\begin{rema}
A similar inequality was proved in \cite[Theorem 1.4]{KW2} for h-convex hypersurfaces $\Sigma$ in $\H^n$ enclosing a bounded domain $\Omega$ that contains the origin.
 Specifically,  for any non-decreasing convex $C^2$ positive function $f$, it states:
\begin{equation}\label{KW2}
\int_{\Sigma} f(\Phi) H_1 d\mu-\int_{\Omega}\left(f(\Phi)-f'(\Phi)\lambda'\right)dv\geq \tilde\chi\circ \xi^{-1} (|\Sigma|),
\end{equation}
where $\tilde\chi(r)=\int_{\p B_r}f(\Phi)H_1 d\mu-\int_{B_r}\left(f(\Phi)-f'(\Phi)\lambda'\right)dv$, and  $\xi(r)=\omega_{n-1}\lambda^{n-1}(r)=\omega_{n-1}(\sinh r)^{n-1}$.
The left hand side of \eqref{KW2}  is greater than or equal to that of \eqref{W3} as the term $\int_{\Omega}f'(\Phi)\lambda' dv$ is non-negative. Moreover,  \eqref{W3} holds for static convex hypersurfaces,  a broader class than h-convex ones.
Therefore, our inequality \eqref{W3} represents an improvement over \eqref{KW2}.
\end{rema}
\begin{theo}\label{MS}
Suppose that $\Sigma$ is  a strictly convex hypersurface in $\S^n$ enclosing a bounded domain $\Omega$ that contains the origin. Then for any non-decreasing convex $C^2$ positive funciton $f$,
we have
\begin{equation}\label{W4}
\int_{\Sigma} f(\Phi) H_1 d\mu+\int_{\Omega} f(\Phi)dv\geq \chi\circ h^{-1} \left(\int_{\Omega}\lambda'(r)\, dv\right),
\end{equation}
where  $\chi(r)=\int_{\p B_r}f(\Phi)H_1 d\mu+\int_{B_r}f(\Phi)dv$, and  $h(r)=\int_{B_r}\lambda'(r)  dv=\frac 1 n\omega_{n-1}(\sin r)^{n}$.
Moreover, equality in (\ref{W4}) holds if and
only if $\Omega$ is a geodesic ball centered at the origin.
\end{theo}

The inequalities \eqref{W3}, \eqref{W3'} and \eqref{W4} are proved using the locally constraint flow (\ref{flow1}) for $k=1$ in the hyperbolic space ($\v=-1$) and the sphere ($\v=1$), respectively. A key ingredient in the proof  is  the monotonicity of
$$\int_{\Sigma}f(\Phi)H_1d\mu+\v\int_{\Omega}f(\Phi)dv$$ along the flow. By the work in \cite{SX}, it was proved that for strictly convex hypersurfaces in the sphere $\S^n$, the flow
  (\ref{flow1}) preserves strict convexity and converges to a geodesic sphere as  $t\rightarrow \infty.$

The rest of this paper is organized as follows. In Section $2$, we review some basic facts about the $k$-th normalized elementary symmetric function $H_{k}$ and the geometry of hypersurfaces in space forms $N^n(\v)$. In Section $3$, we establish a general variational equation for the curvature integral$\int_{\Sigma}f(\Phi)H_k d\mu$ within the framework of space forms $N^n(\v)$. In Section $4$, we provide a detailed proof of our first main result,  Theorem $1.1$. A key step toward this result is establishing the monotonicity of the weighted curvature integral $\int_{\Sigma}f(\Phi)H_k d\mu$ along the flow \eqref{flow1}.
 Finally, Section $5$ is devoted to proving weighted Minkowski type inequalities in hyperbolic space $\H^n$ and the sphere $\S^n$, specifically addressing Theorem 1.4 and Theorem 1.6.

\section{Preliminaries}
In this section, we first recall some basic definitions and properties of elementary symmetric functions.

Let $\s_k$ be the $k$-th elementary symmetry function $\s_k:\R^{n-1}\to \R$ given by
\[\s_k(\kappa)=\sum_{1\leq i_1<\cdots<i_{k}\leq n-1}\kappa_{i_1}\cdots\kappa_{i_k}\quad  \hbox{ for } \kappa=(\kappa_1, \cdots,\kappa_{n-1})\in \R^{n-1}.\]
For a symmetric matrix $B$, we set
\[
\s_k(B):=\s_k(\kappa(B)),
\]
where $\kappa(B)=(\kappa_1(B),\cdots,\kappa_{n-1}(B))$ is the set of eigenvalues of $B$.
The $k$-th Newton transformation $T_k$ is defined by
$$
\left(T_k\right)_j^i(B)=\frac{\partial \sigma_{k+1}}{\partial B_i^j}(B),
$$
where $B=(B)_i^j$.
We recall the basic formulae about $\s_k$ and $T_k$ (see \cite{Guan, Reilly} for instance):
\begin{eqnarray}\label{sigmak}
\s_k(B)& =&\ds\frac{1}{k!}\d^{i_1\cdots i_k}
_{j_1\cdots j_k}B_{i_1}^{j_1}\cdots
{B_{i_k}^ {j_k}},
\\
(T_k)^{i}_j(B) & =& \ds \frac{1}{k!}\d^{ii_1\cdots i_{k}}
_{j j_1\cdots j_{k}}B_{i_1}^{j_1}\cdots
{B_{i_{k}}^{j_{k}}}\label{Tk},
\end{eqnarray}
where $\d^{i_1i_2\cdots i_{k}}
_{j_1j_2\cdots j_{k}}$ is the generalized Kronecker delta given by
\begin{equation}\label{generaldelta}
 \d^{i_1i_2 \cdots i_k}_{j_1j_2 \cdots j_k}=\det\left(
\begin{array}{cccc}
\d^{i_1}_{j_1} & \d^{i_2}_{j_1} &\cdots &  \d^{i_k}_{j_1}\\
\d^{i_1}_{j_2} & \d^{i_2}_{j_2} &\cdots &  \d^{i_k}_{j_2}\\
\vdots & \vdots & \vdots & \vdots \\
\d^{i_1}_{j_k} & \d^{i_2}_{j_k} &\cdots &  \d^{i_k}_{j_k}
\end{array}
\right).
\end{equation}
The $k$-th G\r{a}rding cone $\Gamma_k^{+}$ is defined by
\begin{equation}\label{Garding}
	\Gamma_k^{+}=\left\{\kappa \in \mathbb{R}^{n-1}: \sigma_j(\kappa)>0,\;\forall\, 1\leq j\leq k\right\} .
\end{equation}
 A symmetric matrix $B$ is said to belong to $\Gamma_k^{+}$ if its eigenvalues $\kappa(B) \in \Gamma_k^{+}$. Let $H_k$ be the normalized $k$-th elementary symmetry function given by
\begin{equation}
	H_k= \frac{\sigma_k}{C_{n-1}^k}.
\end{equation}
We adopt the usual convention $\s_0=H_0=1$ and $\s_k=H_k=0$ for $k\geq n$. We have the following Newton-Maclaurin inequalities.

\begin{lemm} \label{lem} For any integers $1\leq l\leq k\leq n-1$ and $\kappa\in \Gamma_k^+ $,  we have
\begin{equation}\label{N-M}
H_{l}H_{k}\geq H_{l-1} H_{k+1}.
\end{equation}
Equality holds in (\ref{N-M}) at $\kappa$ if and only if $\kappa=c(1,1,\cdots,1)$ for some constant $c>0$.
\end{lemm}

Next we collect some well-known facts on geometry of hypersurfaces in  space forms $N^n(\v)$.

Let $\Sigma \subset N^n(\v))$ be a smooth closed hypersurface in $N^n(\v)$ with induced metric $g$ and the second fundamental form $h$.
In a local coordinate system $\{x^1, \cdots x^{n-1}\}$ of $\Sigma$, we have $g_{ij}=g(\partial_i, \partial_j)$ and $h_{ij}=h(\partial_i, \partial_j).$
The principal curvatures $\kappa=(\kappa_1, \cdots, \kappa_{n-1})$  of $\Sigma$ are the eigenvalues of Weingarten matrix $(h_i^j)=(g^{jk}h_{ik})$, where $(g^{ij})$ is the inverse matrix of $(g_{ij})$.  The $k$-th mean curvature of $\Sigma$ is defined by
\begin{equation}\label{add_sigmak}
\s_k=\s_k(h)=\ds\frac{1}{k!}\d^{i_1\cdots i_k}
_{j_1\cdots j_k}h_{i_1}^{j_1}\cdots
{h_{i_k}^ {j_k}}.
\end{equation}

\begin{lemm}[\cite{Guan}]
Denote $$T_{k-1}^{ij} := \;(T_{k-1})_l^i g^{jl}= \;\frac{\partial \sigma_k}{\partial h_{i}^l} g^{jl},$$ then  we have
\begin{align}
&\sum_{ij} T_{k-1}^{ij} h_{ij}= k\sigma_{k},\\
&\sum_{ij} T_{k-1}^{ij} (h^2)_{ij}= \sigma_1\sigma_{k}-(k+1)\sigma_{k+1},\label{e00}
\end{align}
where $(h^2)_{ij}=g^{kl}h_{ik}h_{jl}$ is the squared second fundamental form.
\end{lemm}

  We present the following formulas for the function $\Phi$  and the support function $u$ which will play a central role in establishing monotonicity properties under the flow (\ref{flow1}).
  \begin{lemm}[\cite{GuanLi2}]
  Let $(\Sigma^{n-1},g)$ be a smooth closed hypersurface in space forms $N^n(\v)$. Then there holds
  \begin{align}\label{nabla}
  &\nabla_i\nabla_j\Phi=\lambda'(r)g_{ij}-uh_{ij},\\
  &\nabla_i(T_{k-1}^{ij} \nabla_j \Phi)=(n-k)\lambda'(r)\sigma_{k-1}-k\sigma_k u=kC_{n-1}^k(\lambda'(r)H_{k-1}-uH_k), \quad k=1,\cdots, n-1 \label{nablak}\\
  & \nabla_i u=h_i^l \nabla_l\Phi,\label{nabla u}
  \end{align}
  where $\nabla$ is the covariant derivative with respect to the metric $g$.
  \end{lemm}
Integrating (\ref{nabla}) and using integration by parts, we obtain the well-known Minkowski formulas in space form $N^n(\v)$
\begin{equation}\label{Minkowski}
\int_{\Sigma} u H_{k}d\mu= \int_{\Sigma} \lambda'(r) H_{k-1} d\mu.
\end{equation}

\section{General evolution equations in the space form}
In this section, we derive a general variational equation for the curvature integral$$\int_{\Sigma}f(\Phi)H_k d\mu$$
weighted by an arbitrary smooth function $f(\Phi)$ in space forms $N^n(\v)$.
We begin by recalling  the following standard variational equations.
\begin{lemm}[\cite{GuanLi2}]\label{flow eqn}
Let $\Sigma_t$ be a smooth family of closed hypersurfaces in  space forms $N^n(\v)$ satisfying a general normal flow
\begin{equation}
\frac{\partial}{\partial t}X=F\nu,
\end{equation}
where $\nu$ is the unit outward normal of $\Sigma_t$ and $F$ is a smooth function on $\Sigma_t$. Then we have the following evolution equations:
\begin{align}
&\partial_t g_{ij}=2Fh_{ij},\label{e01}\\
&\partial_t h_{ij}=-\nabla_i\nabla_j F+F\left((h^2)_{ij}-\v\delta_i^j\right),\label{e02}\\
&\partial_t h^{j}_i=-\nabla_i\nabla^j F-F\left((h^2)^{j}_i+\v\delta_i^j\right),\label{e03}\\
&\frac{d}{dt} \int_{\Sigma_t}H_{k} d\mu_t=\int_{\Sigma_t} \left((n-1-k)H_{k+1}-\v k H_{k-1}\right)F d\mu_t,\qquad k=0,\cdots, n-1, \label{variation}
\end{align}
where $g_{ij},  h_{ij}$ and $h_i^j$ denote the induced metric, second fundamental form, and the Weingarten tensor of the evolving hypersurface $\Sigma_t$ respectively,  and $H_k$ is the $k$-th normalized mean curvature of $\Sigma_t$.
\end{lemm}

\begin{prop}\label{prop1}
Let $\Sigma_t$ be a smooth family of closed hypersurfaces in  space forms $N^n(\v)$ satisfying a general flow
\begin{equation}\label{general flow}
\frac {\partial}{\partial t} X=F\nu,
\end{equation}
where $\nu$ is the outward unit normal of $\Sigma_t$ and $F$ is a smooth function on $\Sigma_t$.
Then for all $k=1,\cdots, n-1$ and $C^2$ function $f$, we have
\begin{align}\label{Vpk}
&\frac{d}{dt}\int_{\Sigma_t} f(\Phi) H_k d\mu_t\nonumber\\
&=\int_{\Sigma_t}\bigg((n-1-k)\,f(\Phi) H_{k+1}+(k+1)f'(\Phi){u} H_k - k\left(f'(\Phi)-\frac{f(\Phi)}{\Phi}\right) \lambda'(r)H_{k-1}\nonumber\\
&\qquad\quad -k \left(\frac{f(\Phi)}{\Phi}\right)H_{k-1} -\frac{1}{C_{n-1}^k}f''(\Phi)T_{k-1}^{ij}\nabla_i\Phi\nabla_j\Phi\bigg)Fd\mu_t.
\end{align}
where $\Phi=\Phi(r)=\int_0^r \lambda(s) ds$ as defined in \eqref{Phi}.
Furthermore, for the case of $k=1$ (the mean curvature case), we have
\begin{align}\label{Vp1}
&\frac{d}{dt}\left(\int_{\Sigma_t} f(\Phi) H_1 d\mu_t+\int_{\Omega_t}\v f(\Phi)dv_t\right)\nonumber\\
&=\int_{\Sigma_t}\bigg((n-2)\,f(\Phi) H_{2}+2 f'(\Phi){u} H_1 - f'(\Phi)\lambda'(r)-\frac{1}{n-1}f''(\Phi)\|\nabla\Phi\|^2\bigg)Fd\mu_t,
\end{align}
where $\Omega_t$ denote the bounded domains enclosed by $\Sigma_t$.
\end{prop}
\begin{proof}
First, note that
\begin{align}
&\frac{\partial}{\partial t} \Phi = \langle\bar\nabla \Phi,\frac{\partial}{\partial t} X \rangle
 = \langle \lambda'(r)\partial r, F\nu\rangle = Fu,\label{e1}\\
 &\frac{\partial}{\partial t} d\mu_t =F\sigma_1 d\mu_t,\label{e2}
\end{align}
where $\bar\nabla$ is the covariant derivative with respect to the ambient space $N^n(\v)$.

Moreover, applying  (\ref{e03}) and (\ref{e00}), we obtain
\begin{align}\label{e3}
\frac{\partial}{\partial t} \sigma_k = &(T_{k-1})^i_j\;\frac{\partial}{\partial t} h^{j}_i\nonumber\\
 =&(T_{k-1})^i_j\bigg(-\nabla_i\nabla^j F-F\left((h^2)^{j}_i+\v\delta_i^j\right)\bigg)\nonumber\\
 =&T_{k-1}^{ij}\bigg(-\nabla_i\nabla_j F-F(h^2)_{ij}-\v F\delta_i^j\bigg)\nonumber\\
 =&-T_{k-1}^{ij} \nabla_i\nabla_j F-F\big(\sigma_1\sigma_k-(k+1)\sigma_{k+1}\big)-\v (n-k)\sigma_{k-1}F.
\end{align}
Combining (\ref{e1})-(\ref{e3}), we derive along the flow (\ref{general flow}) that
\begin{align}
&\frac{d}{d t}\int_{\Sigma_t} f(\Phi) \sigma_k d\mu_t\nonumber\\
&=\int_{\Sigma_t}\bigg(\frac{\partial}{\partial t} f(\Phi)\bigg)\sigma_k d\mu_t+\int_{\Sigma_t}f(\Phi)\frac{\partial\sigma_k}{\partial t} d\mu_t+\int_{\Sigma_t}f(\Phi)\sigma_k (F\sigma_1)d\mu_t\nonumber\\
&=\int_{\Sigma_t}f'(\Phi)uF\sigma_k\!d\mu_t+\!\int_{\Sigma_t}\!f(\Phi)\bigg(\!-\!T_{k-1}^{ij}\nabla_i\nabla_j F
- F\big(\sigma_1\sigma_k\!-\!(k\!+\!1)\sigma_{k+1}\big)-\v F (n-k)\sigma_{k-1} \bigg)d\mu_t\nonumber\\
&\quad + \int_{\Sigma_t} f(\Phi) \sigma_k\sigma_1 F d\mu_t\nonumber\\
&=\int_{\Sigma_t}\bigg(f'(\Phi) u\sigma_k- T_{k-1}^{ij} \nabla_i\nabla_j (f(\Phi)) +(k+1)f(\Phi) \sigma_{k+1}-\v (n-k)f(\Phi)\sigma_{k-1} \bigg) F d\mu_t\nonumber\\
&=\int_{\Sigma_t}\bigg(f'(\Phi) u\sigma_k -f'(\Phi)\big((n-k)\lambda'(r)\sigma_{k-1}-k\sigma_k u\big)-f''(\Phi)T_{k-1}^{ij}\nabla_i\Phi\nabla_j\Phi\nonumber\\
&\qquad\quad  +(k+1)f(\Phi) \sigma_{k+1}-\v (n-k)f(\Phi)\sigma_{k-1} \bigg) F d\mu_t\nonumber\\
&=\int_{\Sigma_t}\bigg((k+1)f'(\Phi)u\sigma_k + (k+1)f(\Phi)\sigma_{k+1}- (n-k)f'(\Phi)\lambda'(r)\sigma_{k-1}-\v (n-k)f(\Phi)\sigma_{k-1} \nonumber\\
&\qquad\quad -f''(\Phi)T_{k-1}^{ij}\nabla_i\Phi\nabla_j\Phi\bigg) F d\mu_t,\nonumber
\end{align}
where we used the divergence-free property of the Newton tensor in the third equality, and (\ref{nablak}) in the fourth equality.
Finally, it  follows from the normalization
$H_k=\frac{\s_k}{C_{n-1}^k}$ that
\begin{align}\nonumber
&\frac{d}{dt}\int_{\Sigma_t} f(\Phi) H_k d\mu_t\nonumber\\
&=\int_{\Sigma_t}\bigg((n-1-k)\,f(\Phi) H_{k+1}+(k+1)f'(\Phi){u} H_k - kf'(\Phi)\lambda'(r)H_{k-1}-\v kf(\Phi)H_{k-1}\nonumber\\
&\qquad\quad -\frac{1}{C_{n-1}^k}f''(\Phi)T_{k-1}^{ij}\nabla_i\Phi\nabla_j\Phi\bigg)Fd\mu_t\label{eqvpk}\\
&=\int_{\Sigma_t}\bigg((n-1-k)\,f(\Phi) H_{k+1}+(k+1)f'(\Phi){u} H_k - k\left(f'(\Phi)-\frac{f(\Phi)}{\Phi}\right) \lambda'(r)H_{k-1}\nonumber\\
&\qquad\quad -k(\lambda'(r)+\v\Phi)\frac{f(\Phi)}{\Phi}H_{k-1} -\frac{1}{C_{n-1}^k}f''(\Phi)T_{k-1}^{ij}\nabla_i\Phi\nabla_j\Phi\bigg)Fd\mu_t.\nonumber
\end{align}
Since for any $\v=-1, 0 $ or $1$, the identity
\begin{equation}\label{relation}
\lambda'(r)+\v\Phi=1,
\end{equation}
always holds, we  thus get the desired evolution equation (\ref{Vpk}).

Applying (\ref{eqvpk}) to the case $k=1$, we have
\begin{align}\label{eVp1}
&\frac{d}{dt} \int_{\Sigma_t} f(\Phi) H_1 d\mu_t\nonumber\\
&=\int_{\Sigma_t}\bigg((n-2)\,f(\Phi) H_{2}+2 f'(\Phi){u} H_1 - f'(\Phi)\lambda'(r)-\v f(\Phi)\nonumber\\
&\qquad\quad-\frac{1}{n-1}f''(\Phi)\|\nabla\Phi\|^2\bigg)Fd\mu_t.
\end{align}
On the other hand, the co-area formula yields
\begin{equation}\label{coarea}
\frac{d}{dt}\int_{\Omega_t}f(\Phi)dv_t= \int_{\Sigma_t}f(\Phi) F d\mu_t.
\end{equation}
Combining (\ref{eVp1}) and (\ref{coarea}) gives (\ref{Vp1}).
\end{proof}

\section{Proof of Theorem $1.1$ }
In this section, we prove Theorem $1.1$ on the weighted Alexandrov-Fenchel type inequality with convex weighting functions for hypersurfaces in $\H^n$.
A key step to the proof involves proving the monotonicity theorem stated below.
\begin{theo}
Suppose that $\Sigma_t$ is a family of static convex hypersurfaces which is the solution of the flow \eqref{flow1} in $\H^n$. Then for $k=2,\cdots, n-1$, and non-decreasing convex $C^2$ positive function $f$ which satisfies
\begin{equation}\label{fc}
f'(s)\geq \frac{k}{k-1}\frac{f(s)}{s},
\end{equation}
 we have the monotonicity of the general weighted $k$-th mean curvature integral
$$\frac{d}{dt}\int_{\Sigma_t} f(\Phi) H_k\leq 0,$$
along the flow \eqref{flow1}. Moreover, equality  holds if and
only if $\Sigma_t$ is a geodesic sphere centered at the origin.
\end{theo}

\begin{proof}
By choosing the speed function $F=\frac {H_{k-1}}{H_k}-\frac{u}{\lambda'}$ in (\ref{Vpk}),  we have
\begin{align}\label{eq1}
&\frac{d}{dt}\int_{\Sigma_t} f(\Phi) H_k d\mu_t\nonumber\\
=& \int_{\Sigma_t}\bigg((n-1-k)\,f(\Phi) H_{k+1}+(k+1)f'(\Phi){u} H_k - k\left(f'(\Phi)-\frac{f(\Phi)}{\Phi}\right)\lambda'H_{k-1}\bigg)\bigg(\frac {H_{k-1}}{H_k}-\frac{u}{\lambda'}\bigg)d\mu_t\nonumber\\
& -\int_{\Sigma_t}k \frac{f(\Phi)}{\Phi} H_{k-1}\bigg(\frac {H_{k-1}}{H_k}-\frac{u}{\lambda'}\bigg)d\mu_t -\int_{\Sigma_t}\frac{1}{C_{n-1}^k}f''(\Phi)T_{k-1}^{ij}\nabla_i\Phi\nabla_j\Phi\bigg(\frac {H_{k-1}}{H_k}-\frac{u}{\lambda'}\bigg)d\mu_t\nonumber\\
=&\int_{\Sigma_t}(n-1-k)\frac{f(\Phi)}{\lambda'} \bigg(\lambda' \frac{H_{k+1}H_{k-1}}{H_k}-u H_{k+1}\bigg)d\mu_t+\int_{\Sigma_t} kf'(\Phi)\frac{u}{\lambda'}(\lambda'H_{k-1}-u H_k)d\mu_t\nonumber\\
&-\int_{\Sigma_t}f'(\Phi)\lambda' H_k\bigg(\frac{H_{k-1}}{H_k}-\frac{u}{\lambda'}\bigg)^2d\mu_t\nonumber\\
&-\int_{\Sigma_t}\left((k-1)f'(\Phi)-k\frac{f(\Phi)}{\Phi}\right)\bigg(\lambda'\frac{H_{k-1}^2}{H_k}
-uH_{k-1}\bigg)d\mu_t-\int_{\Sigma_t}k\frac{f(\Phi)}{\Phi}\frac{1}{\lambda'}\bigg(\lambda'\frac {H_{k-1}^2}{H_k}-uH_{k-1}\bigg)d\mu_t \nonumber\\
&-\frac{1}{C_{n-1}^k}\int_{\Sigma_t}f''(\Phi)T_{k-1}^{ij}\nabla_i\Phi\nabla_j\Phi\bigg(\frac {H_{k-1}}{H_k}-\frac{u}{\lambda'}\bigg)d\mu_t\nonumber\\
\leq &\int_{\Sigma_t}(n-1-k)\frac{f(\Phi)}{\lambda'} \bigg(\lambda' \frac{H_{k+1}H_{k-1}}{H_k}-u H_{k+1}\bigg)d\mu_t+\int_{\Sigma_t} kf'(\Phi)\frac{u}{\lambda'}(\lambda'H_{k-1}-u H_k)d\mu_t\nonumber\\
&-\int_{\Sigma_t}\left((k-1)f'(\Phi)-k\frac{f(\Phi)}{\Phi}\right)\bigg(\lambda'\frac{H_{k-1}^2}{H_k}
-uH_{k-1}\bigg)d\mu_t-\int_{\Sigma_t}k\frac{f(\Phi)}{\Phi}\frac{1}{\lambda'}\bigg(\lambda'\frac {H_{k-1}^2}{H_k}-uH_{k-1}\bigg)d\mu_t \nonumber\\
&-\frac{1}{C_{n-1}^k}\int_{\Sigma_t}f''(\Phi)T_{k-1}^{ij}\nabla_i\Phi\nabla_j\Phi\bigg(\frac {H_{k-1}}{H_k}-\frac{u}{\lambda'}\bigg)d\mu_t\nonumber\\
=&\; I+II+III+IV+V,
\end{align}
where the non-negativity condition $f'(\Phi)\geq 0$ is utilized in the last inequality.
We  now analyze  each term on the right hand side of (\ref{eq1}) individually. First, by applying the Newton-Maclaurin inequality (\ref{N-M})  and the identity (\ref{nablak}), we derive
\begin{eqnarray}\label{eqI}
I&=&\int_{\Sigma_t}(n-1-k)\frac{f(\Phi)}{\lambda'} \bigg(\lambda' \frac{H_{k+1}H_{k-1}}{H_k}-uH_{k+1}\bigg)d\mu_t\nonumber\\
&\leq& \int_{\Sigma_t}(n-1-k)f(\Phi)\frac{1}{\lambda'} \bigg(\lambda' H_k-uH_{k+1}\bigg)d\mu_t\nonumber\\
&=& \frac{1}{(k+1)C_{n-1}^{k+1}}\int_{\Sigma_t}(n-1-k)f(\Phi)\frac{1}{\lambda'}\nabla_i\left(T_{k}^{ij}\nabla_j\Phi\right)d\mu_t\nonumber\\
&=& -\frac{1}{C_{n-1}^{k}}\int_{\Sigma_t} \nabla_i \left(\frac{f(\Phi)}{\Phi}\frac{\Phi}{\lambda'}\right)T_{k}^{ij}\nabla_j\Phi d\mu_t\nonumber\\
&=& -\frac{1}{C_{n-1}^{k}}\int_{\Sigma_t} \nabla_i\left(\frac{f(\Phi)}{\Phi}\right) \frac{\Phi}{\lambda'} T_{k}^{ij}\nabla_i\Phi\nabla_j\Phi d\mu_t-\frac{1}{C_{n-1}^{k}}\int_{\Sigma_t} \frac{f(\Phi)}{\Phi}\nabla_i\left(\frac{\Phi}{\lambda'}\right)T_{k}^{ij}\nabla_j\Phi d\mu_t,
\end{eqnarray}
where we used  the identity $(k+1)C_{n-1}^{k+1}=(n-1-k)C_{n-1}^k$ and integration by parts in the third equality.
A direct computation gives
\begin{eqnarray}\label{eq1'}
\nabla_i\left(\frac{f(\Phi)}{\Phi}\right)&=&\frac{f'(\Phi)\Phi-f(\Phi)}{(\Phi)^2}\nabla_i\Phi.
\end{eqnarray}
Using the identity
$
\lambda'-\Phi=1
$  in $\H^n$, we obtain
\begin{eqnarray}\label{eq2'}
\nabla_i\left(\frac{\Phi}{\lambda'}\right)&=&\frac{(\nabla_i \Phi) \lambda'-\Phi\nabla_i \lambda'}{(\lambda')^2}
=\frac{(\nabla_i \Phi) \lambda'-\Phi\nabla_i \Phi}{(\lambda')^2}\nonumber\\
&=&\frac{1}{(\lambda')^2}\nabla_i\Phi.
\end{eqnarray}
Substituting (\ref {eq1'}) and (\ref {eq2'}) into (\ref{eqI}), using the positive definiteness of the Newton tensor $T_k$ and the assumption
$$f'(s)\geq \frac k {k-1}\frac{f(s)}s \geq \frac{f(s)}s,$$
 we conclude the non-positivity of the first term
\begin{equation}\label{I}
I\leq 0.
\end{equation}
The inequality \eqref{I} is strict unless $\nabla\Phi\equiv 0$ on $\Sigma_t$ , which means that $\Sigma_t$ is a geodesic sphere centered at the origin.

Next, we deal with the term $II$ by multiplying the function $f'(\Phi)\frac{u}{\lambda'}$ on  both sides of (\ref{nablak}) and integrating by parts.
Since $T_{k-1}^{ij}$ is divergence-free,  we deduce that
\begin{align*}
&II=\int_{\Sigma_t} kf'(\Phi)\frac{u}{\lambda'}(\lambda'H_{k-1}-u H_k)d\mu_t=\frac{1}{C_{n-1}^k} \int_{\Sigma_t} f'(\Phi)\frac{u}{\lambda'}\nabla_i(T_{k-1}^{ij}\nabla_j\Phi)d\mu_t\\
=&-\frac{1}{C_{n-1}^k}\int_{\Sigma_t}f''(\Phi)\frac{u}{\lambda'} T_{k-1}^{ij}\nabla_i\Phi\nabla_j\Phi d\mu_t
-\frac{1}{C_{n-1}^k}\int_{\Sigma_t}f'(\Phi)T_{k-1}^{ij}\nabla_i\left( \frac{u}{\lambda'}\right)\nabla_j\Phi d\mu_t.
\end{align*}
Applying the identities (\ref{nabla u}) and (\ref{relation}), we compute
$$
\nabla_i\left(\frac{u}{\lambda'}\right)=\frac{h_i^s(\nabla_s \Phi) \lambda'-u\nabla_i \lambda'}{(\lambda')^2}
=\frac{\left(h_i^s \lambda'- u\delta_i^s\right) \nabla_s\Phi}{(\lambda')^2}.
$$
Thus we get
$$-\frac{1}{C_{n-1}^k}\int_{\Sigma_t}f'(\Phi)T_{k-1}^{ij}\nabla_i\left( \frac{u}{\lambda'}\right)\nabla_j\Phi d\mu_t\leq 0,$$
where we used $\left(h_i^s \lambda'- u\delta_i^s\right)\geq 0$ and $T_{k-1}$ is positively definite, since $\Sigma_t$ is static convex.
This leads to
\begin{equation}\label{eq2}
II\leq -\frac{1}{C_{n-1}^k}\int_{\Sigma_t}f''(\Phi)\frac{u}{\lambda'} T_{k-1}^{ij}\nabla_i\Phi\nabla_j\Phi d\mu_t.
\end{equation}
Therefore,
\begin{equation}\label{eq25}
  II+V\leq -\frac{1}{C_{n-1}^k}\int_{\Sigma_t}f''(\Phi)T_{k-1}^{ij}\frac {H_{k-1}}{H_k}\nabla_i\Phi\nabla_j\Phi d\mu_t.
\end{equation}
For the case $k=2,\cdots, n-1,$  and under the assumption
$$(k-1)f'(s)-k\frac{f(s)}{s}\geq 0,$$
we can use the Newton-MacLaurin inequality (\ref{N-M}) to deal with the third term III as follows:
\begin{align}\label{eq5}
III=&-\int_{\Sigma_t}\left((k-1)f'(\Phi)-k\frac{f(\Phi)}{\Phi}\right)\bigg(\lambda'\frac{H_{k-1}^2}{H_k}
-uH_{k-1}\bigg)d\mu_t\nonumber\\
\leq & -\int_{\Sigma_t}\left((k-1)f'(\Phi)-k\frac{f(\Phi)}{\Phi}\right)(\lambda' H_{k-2}-uH_{k-1})d\mu_t\nonumber\\
=&-\int_{\Sigma_t} \left((k-1)f'(\Phi)-k\frac{f(\Phi)}{\Phi}\right)\frac{1}{(k-1)C_{n-1}^{k-1}}\nabla_i(T^{ij}_{k-2}\nabla_j\Phi)d\mu_t\nonumber\\
=&\frac{1}{C_{n-1}^{k-1}}\int_{\Sigma_t}\left( f''(\Phi)-\frac{k}{k-1}\left(\frac{f(\Phi)}{\Phi}\right)'\right)T^{ij}_{k-2}\nabla_i\Phi \nabla_j\Phi d\mu_t,
\end{align}
where we used (\ref{nablak}) and integration by parts.
By the Newton-MacLaurin inequality (\ref{N-M}), identity (\ref{nablak}) and integration by parts, we derive
\begin{align*}
IV&=-\int_{\Sigma_t}k\frac{f(\Phi)}{\Phi}\frac{1}{\lambda'}\bigg(\lambda'\frac {H_{k-1}^2}{H_k}-uH_{k-1}\bigg)d\mu_t \nonumber\\
&\leq -\int_{\Sigma_t}k\frac{f(\Phi)}{\Phi}\frac{1}{\lambda'}\bigg(\lambda'H_{k-2}-uH_{k-1}\bigg)d\mu_t \nonumber\\
&=-\frac{1}{(k-1)C_{n-1}^{k-1}}\int_{\Sigma_t}k\frac{f(\Phi)}{\Phi}\frac{1}{\lambda'} \nabla_i(T_{ij}^{k-2}\nabla_j\Phi) d\mu_t\\
&=\frac{1}{C_{n-1}^{k-1}}\int_{\Sigma_t}\frac{k}{k-1}\left(\frac{f(\Phi)}{\Phi}\right)'\frac{1}{\lambda'}T^{ij}_{k-2}\nabla_i\Phi \nabla_j\Phi d\mu_t\\
&\quad +\frac{1}{C_{n-1}^{k-1}}\int_{\Sigma_t}\left( \frac{k}{k-1} \frac{f(\Phi)}{\Phi}\right)\nabla_i\left(\frac{1}{\lambda'}\right)T^{ij}_{k-2} \nabla_j\Phi d\mu_t.
\end{align*}
By using again the identy
$
\lambda'-\Phi=1,
$ we obtain
\begin{equation}\label{eq11'}
\nabla_i\left(\frac{1}{\lambda'}\right)= -\frac{\nabla_i \lambda'}{(\lambda')^2}
=-\frac{ \nabla_i \Phi }{(\lambda')^2}.
\end{equation}
Using the positive definiteness of $T_{k-2}$, we have
\begin{align}
&\frac{1}{C_{n-1}^{k-1}}\int_{\Sigma_t}\left( \frac{k}{k-1}\big(\frac{f(\Phi)}{\Phi}\big)\right)\nabla_i\left(\frac{1}{\lambda'}\right)T^{ij}_{k-2} \nabla_j\Phi d\mu_t\nonumber\\
=&-\frac{1}{C_{n-1}^{k-1}}\int_{\Sigma_t}\left( \frac{k}{k-1}\frac{f(\Phi)}{\Phi}\right)\frac{1}{(\lambda')^2}T^{ij}_{k-2} \nabla_i\Phi\nabla_j\Phi d\mu_t\nonumber\\
&\leq 0,
\end{align}
which yields
\begin{align}\label{004}
IV&\leq \frac{1}{C_{n-1}^{k-1}}\int_{\Sigma_t}\frac{k}{k-1}\left(\frac{f(\Phi)}{\Phi}\right)'\frac{1}{\lambda'}T^{ij}_{k-2}\nabla_i\Phi \nabla_j\Phi d\mu_t.
\end{align}
The inequality \eqref{004} is strict unless $\nabla\Phi\equiv 0$ on $\Sigma_t$ , which means that $\Sigma_t$ is a geodesic sphere centered at the origin.
Substituting  (\ref{I}), (\ref{eq25}), (\ref{eq5}) and (\ref{004}) into (\ref{eq1}), we finally arrive at
\begin{align}\label{eq6}
&\frac{d}{dt} \int_{\Sigma_t} f(\Phi) H_k d\mu_t\nonumber\\
\leq&~\frac{1}{C_{n-1}^{k-1}}\int_{\Sigma_t} \frac{k}{k-1}\left(\frac{f(\Phi)}{\Phi}\right)' \left(-1+\frac{1}{\lambda'}\right)T^{ij}_{k-2}\nabla_i\Phi \nabla_j\Phi d\mu_t\nonumber\\
 & -\int_{\Sigma_t} f''(\Phi)\bigg(\frac{1}{C_{n-1}^k}T_{k-1}^{ij}\frac {H_{k-1}}{H_k} -\frac{1}{C_{n-1}^{k-1}}T^{ij}_{k-2}\bigg)\nabla_i\Phi \nabla_j\Phi d\mu_t\nonumber\\
 =&~\frac{1}{C_{n-1}^{k-1}}\int_{\Sigma_t} \frac{k}{k-1}\left(\frac{f(\Phi)}{\Phi}\right)'\left(-\frac{\Phi}{\lambda'}\right)T^{ij}_{k-2}\nabla_i\Phi \nabla_j\Phi d\mu_t\nonumber\\
 & -\int_{\Sigma_t} f''(\Phi)\bigg(\frac{1}{C_{n-1}^k}T_{k-1}^{ij}\frac {H_{k-1}}{H_k} -\frac{1}{C_{n-1}^{k-1}}T^{ij}_{k-2}\bigg)\nabla_i\Phi \nabla_j\Phi d\mu_t\nonumber\\
 &\leq -\int_{\Sigma_t} f''(\Phi)\bigg(\frac{1}{C_{n-1}^k}T_{k-1}^{ij}\frac {H_{k-1}}{H_k} -\frac{1}{C_{n-1}^{k-1}}T^{ij}_{k-2}\bigg)\nabla_i\Phi \nabla_j\Phi d\mu_t,
\end{align}
where we used the positive definiteness of $T_{k-2}$ and the non-negativity
$$\left(\frac{f(s)}{s}\right)'\geq 0,$$
which follows directly from the assumption \eqref{fc} on the function $f$ .

Using the key identity first introduced in \cite{Wu}
\begin{equation}\label{key}
\frac{1}{C_{n-1}^k}T_{k-1}^{ij}\frac {H_{k-1}}{H_k} -\frac{1}{C_{n-1}^{k-1}}T^{ij}_{k-2}
=\frac{H_{k-1}^2}{H_k} g^{js}\frac{\partial}{\partial h_{i}^s}\bigg(\frac{H_k}{H_{k-1}}\bigg),
\end{equation}
we finally conclude that
\begin{align}\label{eq7}
&\frac{d}{dt} \int_{\Sigma_t} f(\Phi) H_k d\mu_t\nonumber\\
\leq&~-\int_{\Sigma_t} f''(\Phi) \frac{H_{k-1}^2}{H_k}\frac{\partial(\frac{H_k}{H_{k-1}})}{\partial h_{i}^j}\nabla_i\Phi \nabla^j\Phi d\mu_t\nonumber\\\
\leq & ~0,
\end{align}
where we used the fact $\frac{H_k}{H_{k-1}}$ is strictly increasing with the shape operator (see e.g. \cite[Theorem 1.4 and Theorem 2.16]{Spruck}) and the convexity assumption $f''(\Phi)\geq 0$ in the last inequality. Moreover, the inequality is strict unless $\nabla\Phi\equiv 0$ on $\Sigma_t$ , which means that $\Sigma_t$ is a geodesic sphere centered at the origin.
\end{proof}

\noindent{\it Proof of Theorem 1.1.}
Let $\Omega_t$ be the bounded domain enclosed by $\Sigma_t$.
By the variational formula  of the quermassintegral \eqref{vari} along the flow (\ref{flow1}), we have for $ 1\leq l\leq k,$
\begin{eqnarray*}
\frac{d}{dt}W_l(\Omega_t)&=&\int_{\Sigma_t} H_l\left(\frac{H_{k-1}}{H_k}-\frac{u}{\lambda'(r)}\right)d\mu\\
&\geq& \int_{\Sigma_t} \frac 1{\lambda'(r)}(\lambda'(r)H_{l-1}-H_l u)d\mu\\
&=& \frac{1}{lC_{n-1}^l}\int_{\Sigma_t}\frac 1 {\lambda'(r)} \nabla_i\left(T_{l-1}^{ij}\nabla_j\Phi\right)d\mu\\
&=& \frac{1}{lC_{n-1}^l}\int_{\Sigma_t}\frac 1 {\left(\lambda'(r)\right)^2}\, T_{l-1}^{ij}\nabla_j\Phi\nabla_i\Phi d\mu\\
&\geq&0,
\end{eqnarray*}
with equality holds if and only if $\Sigma_t$ is a geodesic sphere,
where  we used the Newton-Maclaurin inequality \eqref{N-M} in the first inequality, invoked \eqref{nablak} in the second equality, and applied integration by parts  in the third equality.
We now apply the convergence of the flow \eqref{flow1}.  Let $S_{r_\infty}=\partial B_{r_\infty}$ denote the limiting geodesic sphere.  Then
\begin{equation}\label{4.1}
W_l(\Omega)\leq \lim_{t\rightarrow\infty} W_l(\Omega_t)=W_l(B_{r_\infty})=f_l(r_{\infty}).
\end{equation}
Furthermore, by  Theorem 4.1,
\begin{equation}\label{4.2}
\int_{\Sigma} f(\Phi) H_k d\mu \geq \lim_{t\rightarrow\infty} \int_{\Sigma_t} f(\Phi) H_k d\mu_t=\int_{S_{r_\infty}} f(\Phi) H_k d\mu=\chi_k(r_{\infty}).
\end{equation}
Since both functions $\chi_k:\R_+\rightarrow \R_+ $  and $f_l:\R_+\rightarrow \R_+ $
 are strictly increasing, combining\eqref{4.1} and  \eqref{4.2}, we obtain the desired inequality
\begin{equation}\label{4.9}
\int_{\Sigma} f(\Phi) H_k d\mu\geq \chi_k\circ f_l^{-1}(W_l(\Omega)).
\end{equation}

By the coarea formula and the Newton-Maclaurin inequality \eqref{N-M}, we have
\begin{equation}
\frac{d}{dt}\int_{\Omega_t}\lambda' dv= \int_{\Sigma_t}\lambda'\left(\frac{H_{k-1}}{H_k}- \frac{u}{\lambda'}\right)d\mu_t\geq  \int_{\Sigma_t}\left(\frac{\lambda'}{H_1}- u\right)d\mu_t\geq 0,
\end{equation}
where we used the Heintze-Karcher inequality \cite[Theorem 3.5]{B13} in the last inequality.  This inequality is strict unless $\Sigma_t$ is a geodesic sphere centered at the origin.
Combining with the convergence result of the flow \eqref{flow1},  we deduce that
$$\int_{\Sigma} f(\Phi) H_k d\mu \geq \chi_k(r_{\infty})=\chi_k\circ h^{-1}\left(\int_{B_{r_{\infty}}}\lambda'(r)dv\right)\geq \chi_k\circ h^{-1} \left(\int_{\Omega}\lambda'(r)\right).$$
This proves the inequality \eqref{W0}.
The proof of $l=0$ in \eqref{W2} follows from \eqref{W0} and the weighted geometric inequality
$$\int_{\Omega}\lambda'(r) dv \geq h\circ f_0^{-1}(W_0(\Omega)),$$
which holds for star-shaped domain in $\H^n$ (see \cite[Theorem 1.9]{HL}).

The rigidity of \eqref{W2} (\eqref{W0} resp.) follows from the rigidity of the monotonicity in Theorem 4.1.  This complete the proof of Theorem $1.1$.

\section{Weighted Minkowski type inequalities in hyperbolic space and the sphere }\label{sec+}
In this section, we will focus on the Minkowski type inequaties with general convex weight in  hyperbolic space and the sphere.
\subsection{The hyperbolic space as the ambient space}
We first establish the following monotonicity formula in the hyperbolic space.
\begin{theo}
Suppose that $\Sigma_t$ is a family of static convex hypersurfaces which is the solution of the flow \eqref{flow1} in $\H^n$. Then for any non-decreasing convex $C^2$ positive funciton $f$ which satisfies
\begin{equation}
\left(\frac{f(s)}{s}\right)'\geq 0,
\end{equation}
 we have the monotonicity of the general weighted  mean curvature integral
$$\frac{d}{dt}\left(\int_{\Sigma_t} f(\Phi) H_1 d\mu_t-\int_{\Omega_t} f(\Phi)dv_t\right)\leq 0,$$
along the flow \eqref{flow1}. Moreover, equality  holds if and
only if $\Omega_t$ is a geodesic ball centered at the origin.
\end{theo}
\begin{proof}
First by choosing the speed function $F=\frac {1}{H_1}-\frac{u}{\lambda'}$ in (\ref{Vp1}),  we derive
\begin{align}\label{eqH}
&\frac{d}{dt} \left(\int_{\Sigma_t} f(\Phi) H_1 d\mu_t-\int_{\Omega_t} f(\Phi)dv_t\right)\nonumber\\
=& \int_{\Sigma_t}\bigg((n-2)\,f(\Phi) H_{2}+2 f'(\Phi){u} H_1 - f'(\Phi)\lambda'(r)-\frac{1}{n-1}f''(\Phi)\|\nabla\Phi\|^2\bigg)
\bigg(\frac {1}{H_1}-\frac{u}{\lambda'}\bigg)d\mu_t\nonumber\\
=&\int_{\Sigma_t}(n-2)\frac{f(\Phi)}{\lambda'} \bigg(\lambda' \frac{H_2}{H_1}-u H_{2}\bigg)d\mu_t+\int_{\Sigma_t} f'(\Phi)\frac{u}{\lambda'}(\lambda'-u H_1)d\mu_t\nonumber\\
&-\int_{\Sigma_t}f'(\Phi)\lambda' H_1\bigg(\frac{1}{H_1}-\frac{u}{\lambda'}\bigg)^2d\mu_t-\frac{1}{n-1}\int_{\Sigma_t}f''(\Phi)\|\nabla\Phi\|^2\bigg(\frac {1}{H_1}-\frac{u}{\lambda'}\bigg)d\mu_t\nonumber\\
=&\; I+II+III+IV.
\end{align}
Next, we deal with each term separately.   Clearly,   $f'(\Phi)\geq 0$ implies
\begin{equation}\label{H3}
 III\leq 0.
\end{equation}
By the Newton-Maclaurin inequality (\ref{N-M})  and (\ref{nablak}), we deduce
\begin{eqnarray}\label{eqH1}
I&=&\int_{\Sigma_t}(n-2)\frac{f(\Phi)}{\lambda'} \bigg(\lambda' \frac{H_2}{H_1}-uH_{2}\bigg)d\mu_t\nonumber\\
&\leq& \int_{\Sigma_t}(n-2)f(\Phi)\frac{1}{\lambda'} \bigg(\lambda' H_1-uH_{2}\bigg)d\mu_t\nonumber\\
&=& \frac{1}{2C_{n-1}^{2}}\int_{\Sigma_t}(n-2)f(\Phi)\frac{1}{\lambda'}\nabla_i\left(T_{1}^{ij}\nabla_j\Phi\right)d\mu_t\nonumber\\
&=& -\frac{1}{n-1}\int_{\Sigma_t} \nabla_i \left(\frac{f(\Phi)}{\Phi}\frac{\Phi}{\lambda'}\right)T_{k}^{ij}\nabla_j\Phi d\mu_t\nonumber\\
&=& -\frac{1}{n-1}\left(\int_{\Sigma_t} \nabla_i\left(\frac{f(\Phi)}{\Phi}\right) \frac{\Phi}{\lambda'} T_{1}^{ij}\nabla_i\Phi\nabla_j\Phi d\mu_t+\int_{\Sigma_t} \frac{f(\Phi)}{\Phi}\nabla_i\left(\frac{\Phi}{\lambda'}\right)T_{1}^{ij}\nabla_j\Phi d\mu_t\right)
\end{eqnarray}
Through direct computation, we obtain
\begin{eqnarray}\label{eqH2'}
\nabla_i\left(\frac{f(\Phi)}{\Phi}\right)&=&\frac{f'(\Phi)\Phi-f(\Phi)}{(\Phi)^2}\nabla_i\Phi.
\end{eqnarray}
By using the identity
$
\lambda'-\Phi=1,
$ in hyperbolic space $\H^n, $ we derive
\begin{eqnarray}\label{eqH1'}
\nabla_i\left(\frac{\Phi}{\lambda'}\right)&=&\frac{(\nabla_i \Phi) \lambda'-\Phi\nabla_i \lambda'}{(\lambda')^2}
=\frac{(\nabla_i \Phi) \lambda'-\Phi\nabla_i \Phi}{(\lambda')^2}=\frac{1}{\lambda'^2}\nabla_i\Phi.
\end{eqnarray}
The assumption $\left(\frac{f(s)}{s}\right)'\geq 0$ is equivalent to
$f'(s)s\geq f(s)$ for all $s>0.$
Substituting (\ref {eqH1'}) into (\ref{eqH1}) and using the positive definiteness of $T_1$, we conclude
\begin{equation}\label{eqI'}
I\leq 0.
\end{equation}
The inequality \eqref{eqI'} is strict unless $\nabla\Phi\equiv 0$ on $\Sigma_t$ , which means that $\Sigma_t$ is a geodesic sphere centered at the origin.

Next, applying (\ref{nablak}) for $k=1$ and integrating by parts,  we deduce that
\begin{align*}
&II=\int_{\Sigma_t} f'(\Phi)\frac{u}{\lambda'}(\lambda'-u H_1)d\mu_t=\frac{1}{n-1} \int_{\Sigma_t} f'(\Phi)\frac{u}{\lambda'}\nabla_i(\delta_i^j\nabla_j\Phi)d\mu_t\\
=&-\frac{1}{n-1}\int_{\Sigma_t}f''(\Phi)\frac{u}{\lambda'} \|\nabla\Phi\|^2 d\mu_t
-\frac{1}{n-1}\int_{\Sigma_t}f'(\Phi)\delta_i^j\nabla_i\left( \frac{u}{\lambda'}\right)\nabla_j\Phi d\mu_t.
\end{align*}
Applying identities (\ref{nabla u}) and (\ref{relation}), we compute
$$
\nabla_i\left(\frac{u}{\lambda'}\right)=\frac{h_i^s(\nabla_s \Phi) \lambda'-u\nabla_i \lambda'}{(\lambda')^2}
=\frac{\left(h_i^s \lambda'- u\delta_i^s\right) \nabla_s\Phi}{(\lambda')^2}.
$$
 For a static convex hypersurface $\Sigma_t \subset \H^n$, the matrix
$$\left(h_i^s \lambda'- u\delta_i^s\right)>0$$
is positive definite.  Combining this with the non-negativity of $f'(\Phi),$ we deduce
$$-\frac{1}{n-1}\int_{\Sigma_t}f'(\Phi)\delta_i^j\nabla_i\left( \frac{u}{\lambda'}\right)\nabla_j\Phi d\mu_t\leq 0,$$
which implies
\begin{equation}\label{eqH2}
II\leq -\frac{1}{n-1}\int_{\Sigma_t}f''(\Phi)\frac{u}{\lambda'} \|\nabla\Phi\|^2 d\mu_t.
\end{equation}
Consequently, the combined terms satisfy
\begin{equation}\label{H1}
  II+IV\leq -\frac{1}{n-1}\int_{\Sigma_t}f''(\Phi)\frac {1}{H_1}\|\nabla\Phi\|^2 d\mu_t.
\end{equation}
Substituting  (\ref{H3}), (\ref{eqI'})and (\ref{eqH1}) into (\ref{eqH}), we conclude that
\begin{align}\label{eq6}
&\frac{d}{dt}\int_{\Sigma_t} \left(\int_{\Sigma_t} f(\Phi) H_1 d\mu_t-\int_{\Omega_t} f(\Phi)dv_t\right)\nonumber\\
\leq&~-\frac{1}{n-1}\int_{\Sigma_t}f''(\Phi)\frac {1}{H_1}\|\nabla\Phi\|^2 d\mu_t\nonumber\\\
\leq&~ 0,
\end{align}
where we used $f''(\Phi)\geq 0$ in the last inequality. Moreover, the inequality is strict unless $\nabla\Phi\equiv 0$ on $\Sigma_t$ , which means that $\Sigma_t$ is a geodesic sphere centered at the origin.
We complete the proof.
\end{proof}
\noindent{\it Proof of Theorem 1.4.}   The area of $\Sigma_t$ evolves as
\begin{eqnarray*}
\frac{d}{dt}|\Sigma_t|&=&(n-1)\int_{\Sigma_t} H_1\left(\frac{1}{H_1}-\frac{u}{\lambda'(r)}\right)d\mu_t\\
&=& (n-1)\int_{\Sigma_t} \frac 1{\lambda'(r)}(\lambda'(r)-H_1 u)d\mu\\
&=& \int_{\Sigma_t}\frac 1 {\lambda'(r)} \nabla_i\left(\delta^{i}_j\nabla_j\Phi\right)d\mu_t\\
&=& \int_{\Sigma_t}\frac 1 {\left(\lambda'(r)\right)^2}\, \|\nabla\Phi\|^2 d\mu_t\\
&\geq&0,
\end{eqnarray*}
with equality holds if and only if $\Sigma_t$ is a geodesic sphere,
where we used \eqref{nablak} and integration by parts. Meanwhile, by the coarea formula,  the weighted volume evolves as
\begin{equation}
\frac{d}{dt}\int_{\Omega_t}\lambda' dv= \int_{\Sigma_t}\lambda'\left(\frac{1}{H_1}- \frac{u}{\lambda'}\right)d\mu_t=\int_{\Sigma_t}\left(\frac{\lambda'}{H_1}- u\right)d\mu_t\geq 0,
\end{equation}
where the inequality follows from the Heintze-Karcher inequality in $\H^n$ proved by Brendle \cite{B13}.

Applying the convergence result of the flow \eqref{flow1}, we obtain the inequalities \eqref{W3} and  \eqref{W3'}. The equality cases in \eqref{W3} and  \eqref{W3'} can be treated similarly as in Theorem 1.1.

\subsection{The sphere as the ambient space}

\begin{theo}
Suppose that $\Sigma_t$ is a family of strictly convex hypersurfaces which is the solution of the flow \eqref{flow1} in $\S^n$. Then for any non-decreasing convex $C^2$ positive funciton $f$,
 we have the monotonicity of the general weighted  mean curvature integral
$$\frac{d}{dt}\left(\int_{\Sigma_t} f(\Phi) H_1 d\mu_t+\int_{\Omega_t} f(\Phi)dv_t\right)\leq 0,$$
along the flow \eqref{flow}.  Moreover, equality  holds if and
only if $\Omega_t$ is a geodesic ball centered at the origin.
\end{theo}
\begin{proof}
Choosing $F=\frac {1}{H_1}-\frac{u}{\lambda'}$ in (\ref{Vp1}),  we have
\begin{align}\label{eqH'}
&\frac{d}{dt} \left(\int_{\Sigma_t} f(\Phi) H_1 d\mu_t+\int_{\Omega_t} f(\Phi)dv_t\right)\nonumber\\
=& \int_{\Sigma_t}\bigg((n-2)\,f(\Phi) H_{2}+2 f'(\Phi){u} H_1 - f'(\Phi)\lambda'(r)-\frac{1}{n-1}f''(\Phi)\|\nabla\Phi\|^2\bigg)
\bigg(\frac {1}{H_1}-\frac{u}{\lambda'}\bigg)d\mu_t\nonumber\\
=&\int_{\Sigma_t}(n-2)\frac{f(\Phi)}{\lambda'} \bigg(\lambda' \frac{H_2}{H_1}-u H_{2}\bigg)d\mu_t+\int_{\Sigma_t} f'(\Phi)\frac{u}{\lambda'}(\lambda'-u H_1)d\mu_t\nonumber\\
&-\int_{\Sigma_t}f'(\Phi)\lambda' H_1\bigg(\frac{1}{H_1}-\frac{u}{\lambda'}\bigg)^2d\mu_t-\frac{1}{n-1}\int_{\Sigma_t}f''(\Phi)\|\nabla\Phi\|^2\bigg(\frac {1}{H_1}-\frac{u}{\lambda'}\bigg)d\mu_t\nonumber\\
=&\; I+II+III+IV.
\end{align}
We estimate each term separately.  It follows from    $f'(\Phi)\geq 0$ that
\begin{equation}\label{H3'}
 III\leq 0.
\end{equation}
Using the  Newton-Maclaurin inequality (\ref{N-M}),  (\ref{nablak}) and integration by parts, we have
\begin{eqnarray}\label{eqH1'}
I&=&\int_{\Sigma_t}(n-2)\frac{f(\Phi)}{\lambda'} \bigg(\lambda' \frac{H_2}{H_1}-uH_{2}\bigg)d\mu_t\nonumber\\
&\leq& \int_{\Sigma_t}(n-2)f(\Phi)\frac{1}{\lambda'} \bigg(\lambda' H_1-uH_{2}\bigg)d\mu_t\nonumber\\
&=& \frac{1}{2C_{n-1}^{2}}\int_{\Sigma_t}(n-2)f(\Phi)\frac{1}{\lambda'}\nabla_i\left(T_{1}^{ij}\nabla_j\Phi\right)d\mu_t\nonumber\\
&=& -\frac{1}{n-1}\int_{\Sigma_t} \nabla_i \left(\frac{f(\Phi)}{\lambda'}\right)T_{1}^{ij}\nabla_j\Phi d\mu_t\nonumber\\
&=& -\frac{1}{n-1}\left(\int_{\Sigma_t}\frac{f'(\Phi)}{\lambda'} T_{1}^{ij}\nabla_i\Phi\nabla_j\Phi d\mu_t+\int_{\Sigma_t} \frac{f(\Phi)}{(\lambda')^2}T_{1}^{ij}\nabla_i\Phi\nabla_j\Phi d\mu_t\right),
\end{eqnarray}
where we used $\nabla_i \lambda'=-\nabla_i\Phi$ in the sphere $\S^n$.
Hence the positive definiteness of $T_1$ and the fact that $f'\geq 0,\, f>0$ yield
\begin{equation}\label{eqI1}
I\leq 0.
\end{equation}
The inequality \eqref{eqI1} is strict unless $\nabla\Phi\equiv 0$ on $\Sigma_t$ , which means that $\Sigma_t$ is a geodesic sphere centered at the origin.

Next, applying (\ref{nablak}) for $k=1$ and integrating by parts,  we obtain
\begin{align*}
&II=\int_{\Sigma_t} f'(\Phi)\frac{u}{\lambda'}(\lambda'-u H_1)d\mu_t=\frac{1}{n-1} \int_{\Sigma_t} f'(\Phi)\frac{u}{\lambda'}\nabla_i(\delta_i^j\nabla_j\Phi)d\mu_t\\
=&-\frac{1}{n-1}\int_{\Sigma_t}f''(\Phi)\frac{u}{\lambda'} \|\nabla\Phi\|^2 d\mu_t
-\frac{1}{n-1}\int_{\Sigma_t}f'(\Phi)\delta_i^j\nabla_i\left( \frac{u}{\lambda'}\right)\nabla_j\Phi d\mu_t.
\end{align*}
Using the facts (\ref{nabla u}) and (\ref{relation}), we calculate that
$$
\nabla_i\left(\frac{u}{\lambda'}\right)=\frac{h_i^s(\nabla_s \Phi) \lambda'-u\nabla_i \lambda'}{(\lambda')^2}
=\frac{\left(h_i^s \lambda'+ u\delta_i^s\right) \nabla_s\Phi}{(\lambda')^2}.
$$
For a strictly convex hypersurface $\Sigma_t \subset \S^n$,  the matrix
$$\left(h_i^s \lambda'+u\delta_i^s\right)>0$$
is positive definite.
Combining this with the non-negativity of $f'(\Phi),$ we have
$$-\frac{1}{n-1}\int_{\Sigma_t}f'(\Phi)\delta_i^j\nabla_i\left( \frac{u}{\lambda'}\right)\nabla_j\Phi d\mu_t\leq 0,$$
which yields
\begin{equation}\label{eqH2'}
II\leq -\frac{1}{n-1}\int_{\Sigma_t}f''(\Phi)\frac{u}{\lambda'} \|\nabla\Phi\|^2 d\mu_t.
\end{equation}
Consequently,
\begin{equation}\label{H1}
  II+IV\leq -\frac{1}{n-1}\int_{\Sigma_t}f''(\Phi)\frac {1}{H_1}\|\nabla\Phi\|^2 d\mu_t.
\end{equation}
Substituting  (\ref{H3'}), (\ref{eqI1})and (\ref{eqH1'}) into (\ref{eqH'}), we conclude that
\begin{align}\label{eq6}
&\frac{d}{dt}\int_{\Sigma_t} \left(\int_{\Sigma_t} f(\Phi) H_1 d\mu_t+\int_{\Omega_t} f(\Phi)dv_t\right)\nonumber\\
\leq&~-\frac{1}{n-1}\int_{\Sigma_t}f''(\Phi)\frac {1}{H_1}\|\nabla\Phi\|^2 d\mu_t\nonumber\\\
\leq&~ 0,
\end{align}
where we used $f''(\Phi)\geq 0$ in the last inequality. Moreover, the inequality is strict unless $\nabla\Phi\equiv 0$ on $\Sigma_t$ , which means that $\Sigma_t$ is a geodesic sphere centered at the origin.
\end{proof}
\noindent{\it Proof of Theorem 1.6.}
By the coarea formula,  we have
\begin{equation}
\frac{d}{dt}\int_{\Omega_t}\lambda' dv= \int_{\Sigma_t}\lambda'\left(\frac{1}{H_1}- \frac{u}{\lambda'}\right)d\mu_t=\int_{\Sigma_t}\left(\frac{\lambda'}{H_1}- u\right)d\mu_t\geq 0,
\end{equation}
where the inequality follows from the Heintze-Karcher inequality in $\S^n$ proved by Brendle \cite{B13}.   Therefore, Theorem $1.6$ is deduced from Theorem $5.2$ and the convergence result of the flow \eqref{flow1} in $\S^n$.
\vspace{5mm}

{\bf Acknowledgements}
The author would like to thank  Professor Yong Wei for helpful discussions.
 The research is partially supported by National Key R$\&$D Program of China(No. 2022YFA1005500) and Natural Science Foundation of China under Grant No. 11731001.

\end{document}